\documentclass{article}
\usepackage{subfig}
\usepackage{graphicx} 
\usepackage{amsmath,amsfonts,amssymb,amsthm,mathtools}
\usepackage{microtype}
\usepackage{url,hyperref}
\usepackage{indentfirst}
\hypersetup{colorlinks, linkcolor={red}, citecolor={blue}, urlcolor={black}}
\usepackage{comment}
\usepackage{hyperref}
\usepackage{cleveref}
\usepackage{enumerate,enumitem}
\usepackage{tikz}
\usepackage[margin=1.25in]{geometry}
\usepackage[T1]{fontenc}
\usepackage{adjustbox}

\allowdisplaybreaks

\newcommand\abs[1]{\left\lvert#1\right\rvert}
\newcommand\ceil[1]{\left\lceil#1\right\rceil}
\newcommand\floor[1]{\left\lfloor#1\right\rfloor}
\def\eps{\varepsilon}

\newcommand{\HH}{\mathbb{H}}

\newtheorem{theorem}{Theorem}[section]
\newtheorem{lemma}[theorem]{Lemma}
\newtheorem{proposition}[theorem]{Proposition}
\newtheorem{corollary}[theorem]{Corollary}

\newtheorem{question}[theorem]{Question}

\theoremstyle{definition}
\newtheorem{definition}[theorem]{Definition}

\DeclareMathOperator{\supp}{supp}   
\DeclareMathOperator{\ex}{ex}

\newcommand*{\PP}{\mathbb{P}}

\newcommand{\cF}{\mathcal F}

\newcommand{\bi}{\mathbf{i}}

\title{An improved hypergraph Mantel's Theorem}
\date{}
\author{Daniel I\v{l}kovi\v{c}\thanks{Faculty of Informatics, Masaryk University, Brno, Czech Republic; Mathematics Institute and DIMAP, University of Warwick, Coventry, UK.}\and Jun Yan\thanks{Mathematics Institute, University of Warwick, Coventry, UK. Supported by the Warwick Mathematics Institute Centre for Doctoral Training, and by funding from the UK EPSRC (Grant number: EP/W523793/1).}}

\begin{document}

\maketitle
\begin{abstract}
In a recent paper \cite{CY}, Chao and Yu used an entropy method to show that the Tur\'an density of a certain family $\mathcal{F}$ of $\floor{r/2}$ triangle-like $r$-uniform hypergraphs is $r!/r^r$. Later, Liu \cite{L} determined for large $n$ the exact Tur\'an number $\ex(n,\cF)$ of this family, and showed that the unique extremal graph is the balanced complete $r$-partite $r$-uniform hypergraph. These two results together can be viewed as a hypergraph version of Mantel's Theorem. In this paper, building on their methods, we improve both of these results by showing that they still hold with a subfamily $\cF'\subset\cF$ of size $\ceil{r/e}$ in place of $\cF$.
\end{abstract}

\section{Introduction}

Given a family of $r$-uniform hypergraphs $\mathcal{F}$, an $r$-uniform hypergraph $H$ is said to be $\mathcal{F}$\textit{-free} if $H$ does not contain any $F\in\mathcal{F}$ as a subgraph. The famous Tur\'{a}n problem studies the following two quantities. 
\begin{itemize}
    \item The \textit{Tur\'{a}n number} $\ex(n,\mathcal{F})$, which is the maximum number of edges in an $\mathcal{F}$-free $r$-uniform hypergraph on $n$ vertices. 
    \item The \textit{Tur\'{a}n density} $\pi(\mathcal{F})$, which is defined to be $\lim_{n\to\infty}\ex(n,\mathcal{F})/\binom nr$.
\end{itemize}
One can show using an averaging argument that $\ex(n,\mathcal{F})/\binom nr$ is decreasing in $n$, which implies that the limit above always exists, so $\pi(\mathcal{F})$ is well-defined.  

In the case when $r=2$, or equivalently on graphs, the Tur\'{a}n problem is well-understood. For $\chi\geq 2$, define the \textit{Tur\'{a}n graph} $T(n,\chi)$ to be the balanced complete $\chi$-partite graph on $n$ vertices. The earliest and most basic result in this area is Mantel's Theorem, which states that for the triangle graph $K_3$, $\pi(K_3)=1/2$ and $\ex(n,K_3)=\floor{n^2/4}$, with $T(n,2)$ being the unique extremal graph. Tur\'{a}n's Theorem generalises this to all cliques, asserting that $\pi(K_{\chi+1})=1-1/\chi$ and $\ex(n,K_{\chi+1})=|E(T(n,\chi))|$, with $T(n,\chi)$ being the unique extremal graph. For non-bipartite graphs, the Tur\'{a}n problem is solved asymptotically by the Erd\H{o}s-Stone Theorem, which states that if $\chi(F)=\chi+1\geq 3$, then $\pi(F)=1-1/\chi$ and $\ex(n,F) = \left(1-1/\chi+o(1)\right)\binom{n}{2}$, with $T(n,\chi)$ being an asymptotic extremal graph. For a general bipartite graph $F$, the Erd\H{o}s-Stone theorem shows that $\pi(F)=0$ and $\ex(n,F)=o(n^2)$, but the exact asymptotic behaviour of $\ex(n,F)$ remains an interesting open problem. 

In contrast, when $r\geq 3$, the hypergraph Tur\'an problem is notoriously difficult with very few known results. As an example, the Tur\'an density is still unknown even for $K_4^{(3)}$, the complete $3$-uniform hypergraph on 4 vertices. For a comprehensive overview on known results in hypergraph Tur\'an problems, we refer the readers to Keevash's survey \cite{Kee11}.

For every $r\geq 2$ and every family $\mathcal{F}$ of $r$-uniform hypergraphs, a general result states that $\pi(\mathcal{F})\geq r!/r^r$ if and only if no $F\in\mathcal{F}$ is $r$-partite, and $\pi(F)=0$ otherwise. Noting that $\pi(K_3)=1/2=2!/2^2$, a natural avenue of research is to generalise Mantel's Theorem to $r\geq 3$, by finding a family $\mathcal{F}$ of hypergraph analogues of the triangle that satisfies $\pi(\mathcal{F})=r!/r^r$. This is further motivated by a major open conjecture of Erd\H{o}s that states that for every $r\geq 3$, there exists $\eps_r>r!/r^r$ such that there is no family $\mathcal{F}$ of $r$-uniform hypergraphs with $\pi(\mathcal{F})\in(r!/r^r,\eps_r)$. It is known \cite{FR84} that this conjecture is true if and only if there exists a finite family $\mathcal{F}$ of $r$-uniform hypergraphs with $\pi(\mathcal{F})=r!/r^r$ but the blowup density of every $F\in\mathcal{F}$ is larger than $r!/r^r$. As such, a big obstacle towards Erd\H os' conjecture is potentially the lack of characterisations, or even examples, of families satisfying $\pi(\mathcal{F})=r!/r^r$.

One attempt at generalising Mantel's Theorem to hypergraphs uses the $r$-uniform hypergraph $\mathbb{T}_r$ with vertex set $[2r-1]$ and edge set (see Figure \ref{Tr})
\[\{\{1,2,\ldots,r\},\{1,2,\ldots,r-1,r+1\},\{r,r+1,\ldots,2r-1\}\}.\]
Let $T^r(n)$ be the balanced complete $r$-partite $r$-uniform hypergraph on $n$ vertices. Note that $\mathbb{T}_2=K_3$ and $T^2(n)=T(n,2)$. In \cite{FF83}, Frankl and F\"uredi showed that $\pi(\mathbb{T}_3)=3!/3^3=2/9$, and for large $n$ the unique extremal graph is $T^3(n)$. In \cite{Pik08}, Pikhurko showed that $\pi(\mathbb{T}_4)=4!/4^4=3/32$, and for large $n$ the unique extremal graph is $T^4(n)$. However, Frankl and F\"uredi \cite{FF89} showed that $\pi(\mathbb{T}_5)>5!/5^5$ and $\pi(\mathbb{T}_6)>6!/6^6$, so forbidding $\mathbb{T}_r$ itself does not generalise Mantel's Theorem in general. Determining $\pi(\mathbb{T}_r)$ for $r\geq 7$ remains an open problem. 

Recently, Chao and Yu found the following generalisation of Mantel's Theorem in their breakthrough paper \cite{CY}. For every $1\leq i\leq\floor{r/2}$, define the \textit{$(r-i,i)$-tent} $\Delta_{(r-i,i)}$ to be the $r$-uniform hypergraph with vertex set $[2r-1]$ and edge set (see Figure \ref{tent})
\[\{\{1,2,\ldots,r\},\{1,\ldots,i,r+1,\ldots,2r-i-1,2r-1\},\{i+1,\ldots,r,2r-i,\ldots,2r-1\}\}.\] Note that $\Delta_{(r-1,1)}\cong\mathbb{T}_r$. For $1\leq k\leq\floor{r/2}$, let $\mathcal{F}_{r,k}=\{\Delta_{(r-i,i)}\mid1\leq i\leq k\}$. 


\begin{figure}%
    \centering
    \subfloat[\centering $\mathbb{T}_r\cong\Delta_{(r-1,1)}$]{{\includegraphics[width=5cm]{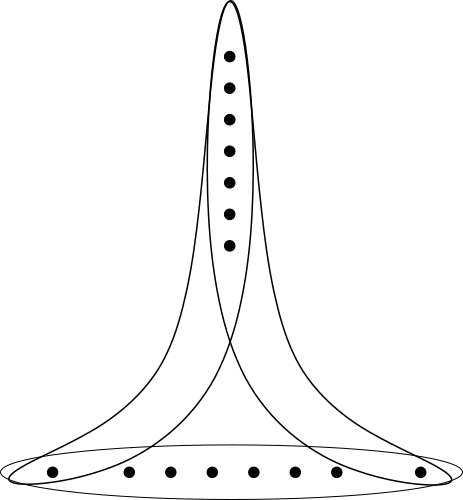}}\label{Tr}}%
    \\
    \subfloat[\centering Tent $\Delta_{(5,3)}$]{{\includegraphics[width=5cm]{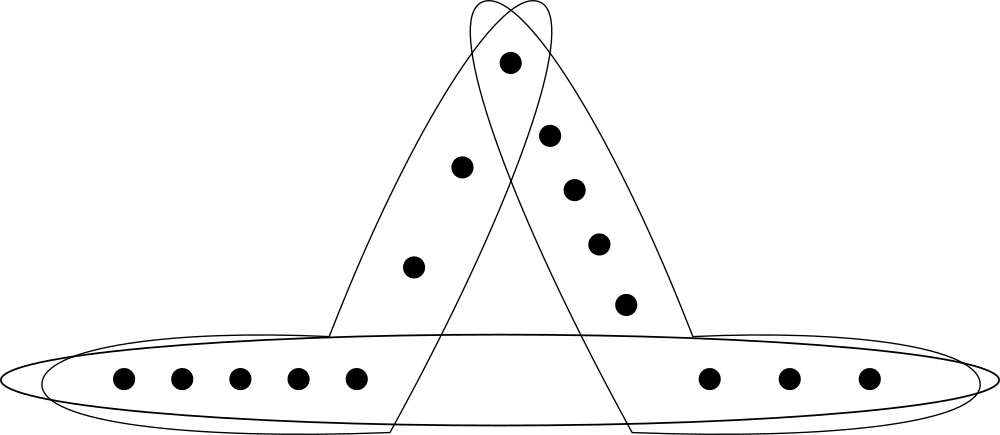}}\label{tent}}%
    \quad
    \subfloat[\centering Partial tent $\Delta^p_{(5,3)}$]{{\includegraphics[width=5cm]{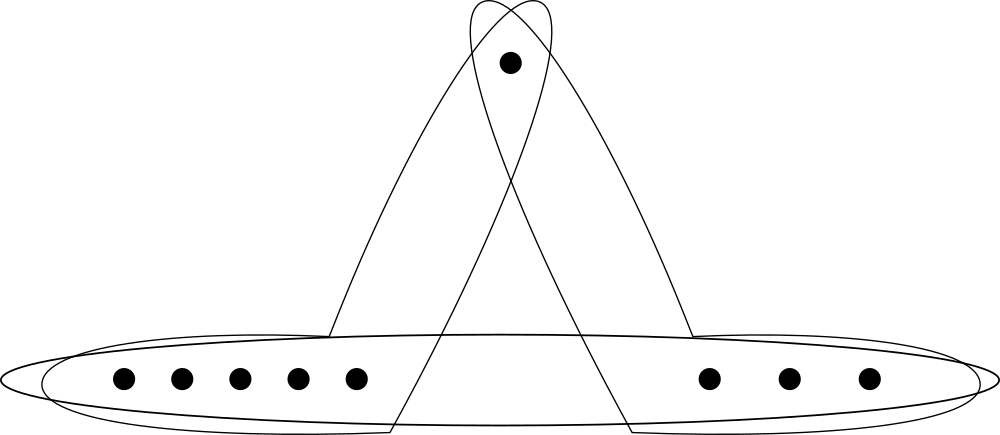}}\label{ptent}}%
    \caption{Examples of triangle-like hypergraphs}
\end{figure}

\begin{theorem}[{\cite[Theorem~1.4]{CY}}]\label{thm:CYmain}
For every $r\geq 2$, $\pi(\mathcal{F}_{r,\floor{r/2}})=r!/r^r$.
\end{theorem}
Chao and Yu's proof of Theorem \ref{thm:CYmain} used a novel entropy method, which reduced this Tur\'an density problem to an optimisation problem whose optimal value is an upper bound on $\pi(\mathcal{F}_{r,\floor{r/2}})$. They then showed that the optimal value is $r!/r^r$, which combined with the fact that no $F\in\mathcal{F}_{r,\floor{r/2}}$ is $r$-partite gives the result.

Soon after, using a powerful framework developed by Liu, Mubayi and Reiher \cite{LMR23unif} and refined by Hou, Liu and Zhao \cite{HLZ24} to prove strong stability results in Tur\'an problems, Liu \cite{L} proved the following exact version of Theorem \ref{thm:CYmain}.  
\begin{theorem}[{\cite[Theorem~1.2]{L}}]\label{thm:Lmain}
For every $r\geq2$ and every sufficiently large $n$, $\ex(n,\mathcal{F}_{r,\floor{r/2}})=|E(T^r(n))|$, with $T^r(n)$ being the unique extremal construction. 
\end{theorem}

In this paper, we improve both Theorem \ref{thm:CYmain} and Theorem \ref{thm:Lmain}. First, we show that forbidding a smaller family of tents still gives Tur\'an density $r!/r^r$. Note that we start with $r=4$ as $r=3$ does not satisfy $\ceil{r/e}\leq\floor{r/2}$.
\begin{theorem}\label{thm:maindensity}
For every $r\geq4$, let $k=\ceil{r/e}$, then $\pi(\mathcal{F}_{r,k})=r!/r^r$.
\end{theorem}
Our proof follows Chao and Yu's entropy method, with most of the effort devoted to showing that the corresponding optimisation problem obtained using $\mathcal{F}_{r,k}$ instead of $\mathcal{F}_{r,\floor{r/2}}$ still has the same optimal value $r!/r^r$. Moreover, we show that the value $k=\ceil{r/e}$ cannot be improved using this method as the corresponding optimal value can exceed $r!/r^r$ if $k<\ceil{r/e}$.

We then proceed as in \cite{L} to transform this density result into an exact Tur\'an number result.
\begin{theorem}\label{thm:mainturan}
For every $r\geq4$, let $k=\ceil{r/e}$, then for every sufficiently large $n$, $\ex(n,\mathcal{F}_{r,k})=|E(T^r(n))|$, with $T^r(n)$ being the unique extremal construction.
\end{theorem}

The rest of this paper is organized as follows. In Section \ref{sec:prelim}, we introduce the necessary definitions related to Chao and Yu's entropy method and recall several of their key lemmas from \cite{CY}. This culminates in Section \ref{sec:reduction} where we adapt their method to reduce the Tur\'an density problem to an optimisation problem. In Section \ref{sec:main}, we solve this optimisation problem and prove Theorem \ref{thm:maindensity}. In Section \ref{sec:turan}, we deduce Theorem \ref{thm:mainturan} from Theorem \ref{thm:maindensity} using the machinery developed in \cite{HLZ24}, \cite{L} and \cite{LMR23unif}. Finally, in Section \ref{sec:conclusion}, we mention several consequences of our results and a few related open problems.

\section{The entropy method}\label{sec:prelim}

In this section, we give an overview of the entropy method introduced by Chao and Yu in~\cite{CY} to prove Theorem \ref{thm:CYmain}. For more details and proofs, we refer the readers to \cite{CY}. We begin by defining (Shannon) entropy.
\begin{definition}[Entropy]
    Let $X,Y,X_1,\ldots,X_n$ be any discrete random variables.
    \begin{itemize}
        \item The \textit{support} of $X$ is the set of values $x$ satisfying $\mathbb{P}(X=x)>0$, and is denoted by $\supp(X)$.
        \item The \textit{(Shannon) entropy} of $X$ is defined to be \[\HH(X)=-\sum_{x\in\supp(X)}\PP(X=x)\log_2\PP(X=x).\]
        \item The entropy of the random tuple $(X_1,\ldots,X_n)$ is denoted by $\HH(X_1,\ldots,X_n)$.
        \item The \textit{conditional entropy} of $X$ given $Y$ is defined to be 
        \begin{align*}
        \HH(X\mid Y)&=\sum_{y\in\supp(Y)}\PP(Y=y)\HH(X\mid Y=y)\\&=-\sum_{(x,y)\in \supp(X,Y)}\PP(X=x,Y=y)\log_2\left(\frac{\PP(X=x,Y=y)}{\PP(Y=y)}\right).
        \end{align*}
    \end{itemize}
\end{definition}

The following proposition collects several useful properties of entropy. 
\begin{proposition}\label{prop:entropyproperty}
Let $X,Y,X_1,\ldots,X_n$ be any discrete random variables.
\begin{itemize}
    \item $\HH(X\mid Y)=\HH(X,Y)-\HH(Y)$.
    \item (Chain rule) $\HH(X_1,\dots,X_n)=\HH(X_1)+\HH(X_2\mid X_1)+\dots+\HH(X_n\mid X_1,\dots,X_{n-1})$.
    \item (Uniform bound) $\HH(X)\leq \log_2\abs{\supp(X)}$, with equality if and only if $X$ is uniform.
    \item (Subadditivity) $\HH(X,Y)\leq \HH(X)+\HH(Y)$.
    \item (Dropping condition) $\HH(X\mid Y)\leq \HH(X)$.
\end{itemize}
\end{proposition}

\subsection{Mixture and the mixture bound}
In this subsection, we define and study the \emph{mixture} of random variables. We are particularly interested in mixtures of random variables whose supports do not overlap too much.
\begin{definition}[Mixture]
    Let $X_1,\dots,X_n$ be discrete random variables. Let $a\in\mathbb{Z}^+$ and let $w_1,\dots,w_n\geq 0$ satisfy $\sum_{i=1}^n w_i=1$.
    \begin{itemize}
    \item $X_1,\dots,X_n$ have \emph{$(a+1)$-wise disjoint supports} if for any element $x\in \cup_{i=1}^n \supp (X_i)$, there are at most $a$ indices $i\in[n]$ such that $x\in\supp (X_i)$. 
    \item The \emph{mixture} of $X_1,\dots,X_n$ with \textit{weight} $w_1,\dots, w_n$ is the random variable $Z$ obtained by first picking an independent random index $\bi$ with probability $\PP(\bi =i)=w_i$ for every $i\in[n]$, then setting $Z=X_{\bi}$.
    \end{itemize}
\end{definition}

The following lemma from~\cite{CY} shows that for given discrete random variables with $(a+1)$-wise disjoint supports, one of their mixtures has large entropy.
\begin{lemma}[Mixture bound, {\cite[Lemma~3.9]{CY}}]\label{lemma:mix}
    Let $X_1,\dots,X_n$ be discrete random variables with $(a+1)$-wise disjoint supports. Then there exists a mixture of $X_1,\dots,X_n$, say $Z$, such that
    \[\sum_{i=1}^n 2^{\HH(X_i)}\leq a2^{\HH(Z)}.\]
\end{lemma}

\subsection{Blowup density and entropic density}
In this subsection we define two related notions of density. We begin with the following well-known notions of homomorphism and blowup density, and two useful results in the study of hypergraph Tur\'an problems. 
\begin{definition}[Homomorphism]
Let $G,H$ be $r$-uniform hypergraphs and let $\mathcal F$ be a family of $r$-uniform hypergraphs. 
\begin{itemize}
    \item A \textit{homomorphism} from $G$ to $H$ is a map $f:V(G)\to V(H)$ such that for every edge $e$ in $G$, the image $f(e)$ is also an edge in $H$. 
    \item $H$ is $\cF$\textit{-hom-free} if there is no homomorphism from any $F\in\cF$ to $H$. 
\end{itemize}
\end{definition}
\begin{definition}[Blowup density]
Let $H=(V,E)$ be an $r$-uniform hypergraph. 
\begin{itemize}
    \item The \textit{Lagrangian} of $H$, denoted by $L(H)$, is defined by \[L(H)=\max\left\{\sum_{e\in E}\prod_{v\in e}x_v\hspace{1.5mm}\middle\vert\hspace{1.5mm}x_v\geq0\text{ for every } v\in V\text{ and }\sum_{v\in V}x_v=1\right\}.\] 
    \item The \textit{blowup density} of $H$, denoted by $b(H)$, is defined to be $r!L(H)$.
\end{itemize}
\end{definition}
\begin{lemma}\label{lemma:homdensity}
Let $\mathcal{F}$ be a family of $r$-uniform hypergraphs, then $\pi(\mathcal{F})=\sup\{b(H)\mid H\text{ is }\mathcal{F}\text{-hom-free}\}$.
\end{lemma}
\begin{corollary}\label{lemma:densityineq}
Let $\cF, \mathcal{F}'$ be two families of $r$-uniform hypergraphs. If for every $F'\in\cF'$, there exists $F\in\cF$ such that there exists a homomorphism from $F$ to $F'$, then $\pi(\cF)\leq\pi(\cF')$.
\end{corollary}

To study the blowup density using entropy, Chao and Yu defined the following concept of entropic density in \cite{CY}. Here, a tuple $(X_1,\dots,X_r)$ of random variables is \textit{symmetric} if the distributions of $(X_{\sigma(1)},\dots,X_{\sigma(r)})$ are the same for all permutations $\sigma$ of $[r]$. In particular, the distributions of $X_1,\ldots,X_r$ are all the same.
\begin{definition}[Entropic density, ratio sequence]
    Let $H$ be an $r$-uniform hypergraph.
    \begin{itemize}
        \item A tuple of random vertices $(X_1,\dots,X_r)\in V(H)^r$ is a \emph{random edge with uniform ordering} on $H$ if $(X_1,\dots,X_r)$ is symmetric and $\{X_1,\dots,X_r\}$ is always an edge of $H$.
        \item The \emph{entropic density} of $H$, denoted by $b_{\textup{entropy}}(H)$, is the maximum value of $2^{\HH(X_1,\ldots, X_r)-r\HH(X_1)}$ over all random edges $(X_1,\ldots, X_r)$ with uniform ordering on $H$.
        \item The \emph{ratio sequence} of a random edge $(X_1,\dots,X_r)$ with uniform ordering is the sequence $(x_1,\ldots,x_r)$ given by $x_i=2^{\HH(X_i\mid X_{i+1},\dots,X_n)-\HH(X_i)}$ for each $i\in [r]$.
    \end{itemize}
\end{definition}

Note that the maximum in the definition of $b_{\textup{entropy}}(H)$ always exists, as the space of random edges with uniform ordering on $H$ is compact. Also, it follows from Proposition~\ref{prop:entropyproperty} that the ratio sequence satisfies $0<x_1\leq\dots\leq x_r=1$ and $\prod_{i=1}^rx_i=2^{\HH(X_1,\ldots, X_r)-r\HH(X_1)}$.

In~\cite{CY}, Chao and Yu proved the following results, which establishes the equivalences between the aforementioned densities.

\begin{proposition}[{\cite[Proposition~5.4]{CY}}]\label{prop:entropic-lagrangian}
    For every $r$-uniform hypergraph $H$, $b_{\textup{entropy}}(H) = b(H)$.
\end{proposition}

\begin{corollary}[{\cite[Corollary~5.6]{CY}}]\label{cor:equiv-to-entropy}
    For every family $\cF$ of $r$-uniform hypergraphs, $\pi(\cF)$ is the supremum of $2^{\HH(X_1,\ldots, X_r)-r\HH(X_1)}$ for any random edge $(X_1,\ldots, X_r)$ with uniform ordering on any $\cF$-hom-free $r$-uniform hypergraph $H$.
\end{corollary}

\subsection{Partial hypergraphs}\label{sec:partial-hypergraph}
In this subsection, we recall Chao and Yu's \cite{CY} definition of partial hypergraphs and collect several of their useful results.

\begin{definition}[Partial hypergraphs]
\mbox{}
\begin{itemize}
    \item A hypergraph $F$ is a \emph{simplicial complex} if for any edge $e\in E(F)$ and any $e'\subset e$, we have $e'\in E(F)$ as well. In particular, every simplicial complex is uniquely determined by its set of maximal edges.
    \item A \textit{partial $r$-hypergraph} is a simplicial complex in which every edge contains at most $r$ vertices.
    \item A \textit{homomorphism} from a partial $r$-hypergraph $F$ to an $r$-uniform hypergraph $H$ is a map $f:V(F)\to V(H)$ such that for any edge $e\in E(F)$, $f$ is injective on $e$ and $f(e)$ is contained in some edge in $E(H)$.
    \item The \textit{extension} of a partial $r$-hypergarph $F$, denoted by $\widetilde{F}$, is the $r$-uniform hypergraph obtained as follows. Let $E'$ be the set of maximal edges in $E(F)$. Then, extend each edge in $E'$ to an edge of size $r$ by adding new extra vertices, with distinct edges not sharing any new vertices.
\end{itemize}
\end{definition}

As an example, for $1\leq i\leq\floor{r/2}$, define $\Delta_{(r-i,i)}^p$ to be the partial $r$-hypergraph with vertex set $[r+1]$ whose set of maximal edges is (see Figure \ref{ptent}) \[\{\{1,2,\ldots,r\},\{1,2,\ldots,i,r+1\},\{i+1,i+2,\ldots,r+1\}\}.\] Then, it is easy to verify that the extension of $\Delta_{(r-i,i)}^p$ is $\Delta_{(r-i,i)}$.

The following result shows that for any partial $r$-hypergraph $F$ and any $r$-uniform hypergraph $H$, a homomorphism from $F$ to $H$ is essentially the same as a homomorphism from $\widetilde{F}$ to $H$. This is helpful later on as instead of considering homomorphisms from $\Delta_{(r-i,i)}$, we can consider homomorphisms from $\Delta_{(r-i,i)}^p$ instead, which are easier to describe.
\begin{proposition}[{\cite[Proposition~6.2]{CY}}]\label{prop:partial-tent-hom}
    Let $F$ be a partial $r$-hypergraph, and let $H$ be an $r$-uniform hypergraph. Then there is a homomorphism from $F$ to $H$ if and only if there is a homomorphism from $\widetilde{F}$ to $H$.
\end{proposition}

Motivated by a method of Szegedy \cite{Sze15}, Chao and Yu showed that one can sample random homomorphisms from some tree-like structures with high entropy. The relevant definitions and their result is as follows. 

\begin{definition}[Partial forest and forest sequence]
Let $F$ be a partial $r$-hypergraph and let $<$ be a linear order on $V(F)$.
\begin{itemize}
    \item For any vertex $v\in V(F)$, let $M_{F,<}(v)$ be the set of edges in $F$ whose maximum vertex under $<$ is $v$.
    \item $F$ is a \emph{partial forest} with respect to $<$ if for every $v\in V(F)$, there is exactly one edge in $M_{F,<}(v)$ that is maximal in $M_{F,<}(v)$. This edge is denoted by $e_v$.
    \item The \emph{forest sequence} of a partial forest $F$ with respect to $<$ is the sequence $(f_1,\ldots, f_{r})$ where for each $i\in[r]$, $f_{i}$ is the number of vertices $v\in V(F)$ with such that $e_v$ contains exactly $i$ vertices. 
\end{itemize}
\end{definition}

\begin{lemma}[{\cite[Lemma~6.5]{CY}}]\label{lemma:sample-tree}
    Let $(X_1,\ldots, X_r)$ be a random edge with uniform ordering on an $r$-uniform hypergraph $H$, and let $(x_1,\ldots, x_r)$ be its ratio sequence. For any partial forest $F$ with respect to a linear order $<$, if $(f_1,\ldots, f_r)$ is its forest sequence, then one can sample a random homomorphism $(Y_v)_{v\in V(F)}$ from $F$ to $H$ with entropy equal to
    \[|V(F)|\HH(X_1)+\log_2\left(\prod_{i=1}^{r}x_i^{f_{r+1-i}}\right).\]
    Moreover, the random homomorphism can be sampled such that for any edge $e\in E(F)$, the distribution of $(Y_v)_{v\in e}$ is the same as $(X_i)_{r-\abs{e}+1\leq i\leq r}$.
\end{lemma}

\subsection{Reduction to an optimisation problem}\label{sec:reduction}
In this subsection, we put everything together and adapt Chao and Yu's entropy method to reduce the problem of determining $\pi(\mathcal{F}_{r,k})$ to an optimisation problem using the following key lemma.   

\begin{lemma}\label{lemma:alt-tent-ineq}
    Let $1\leq k\leq\floor{r/2}$, let $H$ be an $\mathcal{F}_{r,k}$-hom-free $r$-uniform hypergraph, and let $(X_1,\ldots, X_r)$ be a random edge with uniform ordering on $H$ whose ratio sequence is $0< x_1\leq \cdots\leq  x_r=1$. Then, $x_i + x_j \leq x_{i+j}$ for all $i\leq k$ and $i + j \leq n$.
\end{lemma}
\begin{proof}
    Let $V=\{v_1,\ldots, v_r,w\}$ and equip it with the linear order $v_1<\cdots <v_r<w$. Fixed any $i\leq k$ and $i+j\leq r$. Let $F^{(1)}$ be the partial $r$-hypergraph containing exactly two maximal edges $\{v_1,\ldots, v_r\}$ and $\{v_{i+1},\ldots,v_r, w\}$. Let $F^{(2)}$ be the partial $r$-hypergraph exactly two maximal edges $\{v_1,\ldots, v_r\}$ and $\{v_1,\ldots, v_{r-j}, w\}$. It is easy to check from definition that both $F^{(1)}$ and $F^{(2)}$ are partial forests with respect to $<$. Moreover, if $(f^{(1)}_1,\ldots,f^{(1)}_r)$ is the forest sequence of $F^{(1)}$, then $v_\ell$ contributes 1 to $f^{(1)}_\ell$ for every $\ell\in[r]$, and $w$ contributes 1 to $f^{(1)}_{r+1-i}$. Thus, $f^{(1)}_\ell=1$ for every $\ell\in[r]$ except $f^{(1)}_{r+1-i}=2$. Similarly, if $(f^{(2)}_1,\ldots,f^{(2)}_r)$ is the forest sequence of $F^{(2)}$, then $f^{(2)}_\ell=1$ for every $\ell\in[r]$ except $f^{(2)}_{r+1-j}=2$.

    Let $(Y^{(1)}_v)_{v\in V}$ and $(Y^{(2)}_v)_{v\in V}$ be the random homomorphisms from $F^{(1)}$ and $F^{(2)}$ to $H$ given by Lemma~\ref{lemma:sample-tree}, respectively.
    Note that if some tuple $(y_v)_{v\in V}$ of vertices in $V(H)$ is in the supports of both $(Y^{(1)}_v)_{v\in V}$ and $(Y^{(2)}_v)_{v\in V}$, then $(y_v)_{v\in V}$ corresponds to a homomorphism from $F^{(1)}\cup F^{(2)}$ to $H$.
    However, $\Delta^p_{(r-i,i)}$ is a subgraph of $F^{(1)}\cup F^{(2)}$, so there is a homomorphism from $\Delta^p_{(r-i,i)}$ to $H$ as well. Hence, by Proposition~\ref{prop:partial-tent-hom}, there is a homomorphism from $\Delta_{(r-i,i)}$ to $H$, a contradiction. Therefore, the two random homomorphisms $(Y^{(1)}_v)_{v\in V}$ and $(Y^{(2)}_v)_{v\in V}$ have disjoint support.

    This allows us to apply Lemma~\ref{lemma:mix} with $a=1$ to obtain a mixture $(Z_v)_{v\in V}$ of $(Y^{(1)}_v)_{v\in V}$ and $(Y^{(2)}_v)_{v\in V}$ such that, along with Lemma~\ref{lemma:sample-tree}, it satisfies
    \[2^{\HH((Z_v)_{v\in V})}\geq2^{\HH(Y^{(1)}_v)_{v\in V}}+2^{\HH(Y^{(2)}_v)_{v\in V}}\geq2^{(r+1)\HH(X_1)}(x_i+x_j)\prod_{\ell=1}^rx_\ell.\]
    Moreover, observe that both $F^{(1)}$ and $F^{(2)}$ contain the edges $\{v_1,\ldots, v_r\}$ and $\{v_{i+1},\ldots, v_{r-j},w\}$. Therefore, by Lemma~\ref{lemma:sample-tree}, $(Y^{(1)}_{v_1},\ldots, Y^{(1)}_{v_r})$ and $(Y^{(2)}_{v_1},\ldots, Y^{(2)}_{v_r})$ both have the same distributions as $(X_1,\ldots, X_r)$, so $(Z_{v_1},\ldots, Z_{v_r})$ has the same distribution as $(X_1,\ldots, X_r)$ as well by the definition of mixture. Similarly, $(Z_w, Z_{v_{i+1}},\ldots, Z_{v_{r-j}})$ has the same distribution as $(X_{i+j}, \ldots, X_r)$.
    As a consequence, using Proposition~\ref{prop:entropyproperty} and the definition of ratio sequence, we get
    \begin{align*}
        \HH\left((Z_v)_{v\in V}\right)&\leq\HH(Z_{v_1},\ldots, Z_{v_r})+\HH(Z_w\mid Z_{v_{i+1}},\ldots, Z_{v_{r-j}})\\
        &=\HH(X_1,\ldots, X_{r})+\HH(X_{i+j}\mid X_{i+j+1},\ldots, X_{r})\\
        &=(r+1)\HH(X_1)+\log_2\left(x_{i+j}\prod_{\ell=1}^rx_\ell\right).
    \end{align*}
    
    Combining the two inequalities above gives 
    \[2^{(r+1)\HH(X_1)}(x_i+x_j)\prod_{\ell=1}^rx_\ell\leq2^{(r+1)\HH(X_1)}x_{i+j}\prod_{\ell=1}^rx_\ell,\]
    from which it follows that $x_i+x_j\leq x_{i+j}$.
\end{proof}

With Lemma~\ref{lemma:alt-tent-ineq} in mind, for every $1\leq k\leq\floor{r/2}$, define $\mathcal{X}_{r,k}\subset[0,1]^r$ to be 
\[\mathcal{X}_{r,k}=\left\{0<x_1\leq\cdots\leq x_r=1\mid x_i+x_j\leq x_{i+j}\text{ for every }i\in[k]\text{ and }i\leq j\leq r-i\right\}.\]
We can now reduce the Tur\'an density problem to an optimisation problem over $\mathcal{X}_{r,k}$. 
\begin{proposition}\label{prop:main-entropy-ver}
     Let $1\leq k\leq\floor{r/2}$, and let $H$ be an $\mathcal{F}_{r,k}$-hom-free $r$-uniform hypergraph, then 
    \[b(H)=b_{\textup{entropy}}(H)\leq\max\left\{\prod_{i=1}^rx_i\hspace{1.5mm}\middle|\hspace{1.5mm}(x_1,\ldots,x_r)\in\mathcal{X}_{r,k}\right\}.\]
\end{proposition}
\begin{proof}
     Let $(X_1,\ldots,X_r)$ be any random edge with uniform ordering on $H$ and let $0< x_1\leq \cdots\leq  x_n=1$ be its ratio sequence. Recall that $b(H)=b_{\textup{entropy}}(H)$ by Proposition~\ref{prop:entropic-lagrangian} and $\HH(X_1,\dots,X_r)- r\HH(X_1)=\log_2(\prod_{i=1}^rx_i)$ by Proposition~\ref{prop:entropyproperty}. The statement then follows from the definition of entropic density as $(x_1,\ldots, x_n)\in\mathcal{X}_{r,k}$ by Lemma~\ref{lemma:alt-tent-ineq}.
\end{proof}

\section{Proof of Theorem~\ref{thm:maindensity}}\label{sec:main}
In this section, we prove Theorem~\ref{thm:maindensity} by solving the optimisation problem obtained at the end of Section~\ref{sec:prelim}. Specifically, we prove the following.  
\begin{theorem}\label{thm:main}
For every $r\geq 4$, let $k=\ceil{r/e}$, then 
\[\max\left\{\prod_{i=1}^rx_i\hspace{1.5mm}\middle|\hspace{1.5mm}(x_1,\ldots,x_r)\in\mathcal{X}_{r,k}\right\}=\frac{r!}{r^r},\]
and the unique tuple achieving the maximum is given by $x_i=i/r$ for every $i\in[r]$.
\end{theorem}

We first show that Theorem \ref{thm:main} implies Theorem \ref{thm:maindensity}.
\begin{proof}[Proof of Theorem \ref{thm:maindensity}]
    First, we have $\pi(\cF_{r,k})\geq r!/r^r$ by Lemma \ref{lemma:homdensity} as the $r$-uniform hypergraph consisting of just a single edge does not contain the homomorphic image of any $F\in\cF_{r,k}$, and has blowup density $r!/r^r$. For the reverse inequality, if $H$ is $\cF_{r,k}$-hom-free, then by Proposition~\ref{prop:main-entropy-ver} and Theorem~\ref{thm:main}, we have $b(H)=b_{\textup{entropy}}(H)\leq r!/r^r$.
    Thus, by Lemma~\ref{lemma:homdensity}, $\pi(\cF_{n,k})=\sup\{b(H)\mid H\text{ is }\cF_{r,k}\text{-hom-free}\}\leq r!/r^r$.
\end{proof}

The rest of the section is devoted to proving Theorem~\ref{thm:main}. Observe that the maximum of $\prod_{i=1}^rx_i$ over $\mathcal{X}_{r,k}$ can always be attained as $\mathcal{X}_{r,k}$ is compact. Note that if $k=\floor{r/2}$, then Theorem~\ref{thm:main} follows from its counterpart Theorem 7.1 in \cite{CY}, so for the rest of this section we will assume $k=\ceil{r/e}\leq\floor{r/2}-1$, unless stated otherwise. We begin with some preliminary results and definitions.
\begin{lemma}\label{symmetry}
If $(x_1,\ldots,x_r)\in\mathcal{X}_{r,k}$ maximizes $\prod_{i=1}^rx_i$, then $x_j + x_{n-j} = 1$ for all $j\in[k]$.
\end{lemma}
\begin{proof}
We use induction on $j$. For the base case $j=1$, let $x_{r-1}'=1-x_1\geq x_{r-1}$ and $x_i'=x_i$ for all other $i\in[r]$. The only inequality we need to check to verify $(x_1',\ldots,x_r')\in\mathcal{X}_{r,k}$ is $x_1'+x_{r-1}'\leq x_r'=1$, but this holds from definition. Hence, $x_{r-1}'=x_{r-1}=1-x_1$, as otherwise $\prod_{i=1}^rx_i<\prod_{i=1}^rx_i'$, a contradiction. 

Now let $1<j\leq k$ and assume $x_i+x_{r-i}=1$ for all $i\in[j-1]$. Let $x_{r-j}'=1-x_j\geq x_{r-j}$ and $x_i'=x_i$ for all other $i\in[r]$. The only inequalities we need to check to verify $(x_1',\ldots,x_r')\in\mathcal{X}_{r,k}$ are $x_i'+x_{r-j}'\leq x_{r-j+i}'$ for all $i\in[j]$. $x_j'+x_{r-j}'\leq x_{r}'=1$ from definition. For $i\in[j-1]$, $j-i\in[j-1]$, so by induction hypothesis we have $x_{r-j+i}'=x_{r-j+i}=1-x_{j-i}$, and thus $x_i'+x_{r-j}'=x_i+1-x_j\leq 1-x_{j-i}=x_{r-j+i}'$, as required. Thus, $x_{r-j}'=x_{r-j}=1-x_j$, as otherwise $\prod_{i=1}^rx_i$ is not maximized.
\end{proof}

For all integers $L\leq R$, call the set $\{L,L+1,\cdots,R\}$ an \textit{interval}, and denote it by $[L,R]$.

\begin{definition}
Let $(x_1,\ldots,x_r)\in\mathcal{X}_{r,k}$. Recall that $x_r=1$, let $x_0=0$. Let $0\leq L\leq R\leq r$.
\begin{itemize}
    \item $[L,R]$ is a \textit{uniform interval} if $x_i=x_L +(i-L)x_1$ for all $L \leq i\leq R$. 
    \item $[L,R]$ is a \textit{segment} if it is a maximal uniform interval. In other words, if it is a uniform interval that is not properly contained in any other uniform interval.
    \item The \textit{length} of an interval $[L,R]$ is $R-L+1$, the number of integers it contains.
    \item The unique segment of the form $[0,I]$ is called the \textit{initial segment}.
    \item A segment is \textit{central} if it is contained in $[k+1,r-k-1]$. A segment is \textit{left-crossing} if it contains $[k,k+1]$, is \textit{right-crossing} if it contains $[r-k-1,r-k]$, and is \textit{crossing} if it is left and/or right-crossing. 
    \item A segment is \textit{super} if it has length $I+1$.
\end{itemize}
\end{definition}

The following lemma shows that if $I\leq k-1$, then super segments are exactly the segments with maximum possible length. 
\begin{lemma}\label{<=I+1}
If $I\leq k-1$, then every segment has length at most $I + 1$.
\end{lemma}
\begin{proof}
Let $[L,R]$ be a segment. If $[L,R]=[0,I]$ this is trivial, so assume $L\geq I+1$. Suppose for a contradiction that $[L,R]$ has length at least $I+2$, then $L+I+1\leq R$. Since $[0,I]$ and $[L,R]$ are both segments, we have $x_{I+1}>(I+1)x_1$ and $x_{L+I+1}=x_L+(I+1)x_1$. But as $(x_1,\ldots,x_r)\in\mathcal{X}_{r,k}$ and $I+1\leq k$, we have $x_{L+I+1}\geq x_{L}+x_{I+1}>x_L+(I+1)x_1$, a contradiction.
\end{proof}

In the following, we collect several lemmas about segments. 

\begin{lemma}\label{moveleft}
Let $[L_1,R_1],[L_2,R_2]$ be two segments with $L_1<L_2$. If $i\in [L_1,R_1]$ and $i+j\in [L_2,R_2]$ satisfy $L_2-L_1>j$ and $\min\{i,j\}\leq k$, then $x_{i}+x_{j}<x_{i+j}$.
\end{lemma}
\begin{proof}
From the definition of segments, we have $x_{i+j}=x_{L_2}+(i+j-L_2)x_1$ and $x_{i+j}>x_{L_2-1}+(i+j-L_2+1)x_1$. Similarly, using $i-(i+j-L_2+1)=L_2-j-1\geq L_1$, we have $x_{i}=x_{L_2-j-1}+(i+j-L_2+1)x_1$. Thus, $x_{i+j}-x_{i}>x_{L_2-1}-x_{L_2-j-1}\geq x_{(L_2-1)-(L_2-j-1)}=x_j$, as required. Note that in the last inequality we used that either $j\leq k$, or $L_2-j-1\leq i+j-j-1=i-1\leq k$.
\end{proof}

\begin{lemma}\label{moveright}
Let $[L_1,R_1],[L_2,R_2]$ be two segments with $L_1<L_2$. If $i\in [L_1,R_1]$ and $i+j\in [L_2,R_2]$ satisfy $R_2-R_1>j$ and $\min\{R_1+1,j\}\leq k$, then $x_{i}+x_{j}<x_{i+j}$.
\end{lemma}
\begin{proof}
From the definition of segments, we have $x_{R_1}=x_i+(R_1-i)x_1$ and $x_{R_1+1}>x_i+(R_1-i+1)x_1$. Similarly, using $R_2\geq (i+j)+(R_1-i+1)=j+R_1+1$, we have $x_{j+R_1+1}=x_{i+j}+(R_1-i+1)x_1$. Thus, $x_{i+j}-x_{i}>x_{j+R_1+1}-x_{R_1+1}\geq x_j$, as required. Note that in the last inequality we used that either $j\leq k$, or $R_1+1\leq k$.
\end{proof}

\begin{lemma}\label{segsym}
Suppose $I\leq k-1$. Let $0\leq L\leq R\leq r$.
\begin{itemize}
    \item If $R\leq k-1$, then $[L,R]$ is a segment if and only if $[r-R,r-L]$ is. If $R\leq k$, then $[L,R]$ is a super segment if and only if $[r-R,r-L]$ is. 
    \item If $L\leq k$ and $R\geq r-k$, or in other words if $[L,R]$ is both a left and a right-crossing segment, then $L=r-R$.
    \item If $[L,R]$ is left but not right-crossing, then $r-L$ is the right endpoint of a segment that contains $r-k$. If $[L,R]$ is right but not left-crossing, then $r-R$ is the left endpoint of a segment that contains $k$.
\end{itemize}
\end{lemma}
\begin{proof}
Trivially, $x_0=1-x_r$. By Lemma \ref{symmetry}, $x_i=1-x_{r-i}$ for all $i\in[k]$. We will refer to these properties as symmetry. 

If $R\leq k$, then by symmetry, $[L,R]$ is uniform if and only if $[r-R,r-L]$ is uniform. Moreover, if $R\leq k-1$, then $[L,R]$ is maximally uniform if and only if $[r-R,r-L]$ is, so $[L,R]$ is a (super) segment if and only if $[r-R,r-L]$ is. If $R=k$ and $[L,R]$ is a super segment, then $[r-R,r-L]$ is uniform and has the same length as $[L,R]$, so it is also a super segment. The reverse implication is similar. 

Now suppose $[L,R]$ is both a left and a right-crossing segment. If $L<r-R$, then in particular $[L,k]$ is uniform, so $[r-k,r-L]$ is also uniform by symmetry. But as $r-L>R\geq r-k$, this means that $[L,r-L]$ is a longer uniform interval than $[L,R]$, a contradiction. Similarly, we reach a contradiction if $L>r-R$, so $L=r-R$.

If $[L,R]$ is left but not right-crossing, then $L\leq k$ and $[L,k]$ is uniform, so $[r-k,r-L]$ is also uniform by symmetry. If $L=0$, then $r-L=r$ is trivially the right endpoint of the segment containing $[r-k,r-L]$. If $L>0$, then maximality of $[L,R]$ implies that $x_{L-1}<x_L-x_1$, so symmetry implies that $x_{r-L+1}>x_{r-L}+x_1$, and thus $r-L$ is the right endpoint of the segment containing $[r-k,r-L]$. The case when $[L,R]$ is right but not left-crossing is similar. 
\end{proof}

We now prove a key technical lemma that shows that the initial segment of any maximiser must have length at least $k+1$. This is done by proving that if a maximiser has a shorter initial segment, then there is a small perturbation of it that remains inside $\mathcal{X}_{r,k}$ but has a larger product. 
\begin{lemma}\label{I<=k}
If $(x_1,\ldots,x_r)\in\mathcal{X}_{r,k}$ maximizes $\prod_{i=1}^rx_i$, then $I\geq k$.
\end{lemma}
\begin{proof}
    Suppose for contradiction that $I\leq k-1$. By Lemma \ref{symmetry}, $x_i=1-x_{r-i}$ for all $i\in[k]$. Set $x_0=0=1-x_r$. Let $\eps>0$ be a small constant to be chosen later. We define a new sequence $(x_0',x_1',\ldots,x_r')$ by modifying each $x_i$ by at most $\eps$ as follows. Recall that all super segments have length $I+1\geq 2$. 

    \begin{itemize}
        \item  $x'_I = x_I + \eps$ and $x'_{r-I} = x'_{r-I} - \eps$.
        \item For every super segment $[L,R]\not=[0,I],[r-I,r]$ that is not central, set $x'_L = x_L - \eps$ and $x'_R = x_R + \eps$.
        \item If $[L,R]$ is a left but not right-crossing super segment, additionally set $x'_{r-L}=x_{r-L}+\eps$. 
        \item If $[L,R]$ is a right but not left-crossing super segment, additionally set $x'_{r-R}=x_{r-R}-\eps$. 
        \item For every central super segment $[L,R]$, set $x'_R = x_R + \eps$.
        \item For all other variables $x_i$, set $x'_i = x_i$. 
    \end{itemize}
     Note that from Lemma \ref{segsym} and the definition above, if $x_i'=x_i+\eps$, then $i$ is the right endpoint of a segment, and if $x_i'=x_i-\eps$, then $i$ is the left endpoint of a segment. We now show that for sufficiently small $\eps>0$, $(x_1',\ldots,x_r')\in\mathcal{X}_{r,k}$ and $\prod_{i=1}^rx_i'>\prod_{i=1}^rx_i$, which contradicts the maximality of $\prod_{i=1}^rx_i$ and finishes the proof.

    We first show that $(x_1',\ldots,x_r')\in\mathcal{X}_{r,k}$. Observe from Lemma \ref{segsym} and the definition above that $x_i'+x_{r-i}'=1\leq x_r'$ for all $i\in[k]$. For every $i\in[k]$ and $i\leq j\leq r-1-i$, since $(x_1,\ldots,x_r)\in\mathcal{X}_{r,k}$, we have $x_i+x_j\leq x_{i+j}$. Call such a triple of indices $(i,j, i+j)$ \textit{tight} if $x_i + x_j = x_{i+j}$. Note that if $x_i,x_j,x_{i+j}$ is not tight, then for sufficiently small $\eps>0$ we have $x_i'+x_j'\leq x_{i+j}'$ as well. Therefore, we only need to check the desired inequalities hold on triples that are tight. In the case analysis that follows, we assume $i\in[k]$, $i\leq j\leq r-1-i$ and $(i,j,i+j)$ is tight. For every $*\in[r]$, we denote the left and right endpoints of the unique segment containing $*$ by $L_*$ and $R_*$. 

    \textbf{Case 1.} $x_i'=x_i+\eps$.

    From definition, since $x_i'=x_i+\eps$ and $i\leq k$, $[L_i,i]$ must be a super segment. Then, $(i+j)-L_{i+j}\leq i-L_i$, so $L_{i+j}-L_i\geq j$. If $L_{i+j}-L_i>j$, then the triple $(i, j, i+j)$ is not tight by Lemma \ref{moveleft}, a contradiction. If $L_{i+j}-L_i=j$, then $[L_{i+j},i+j]$ is also a super segment because it has the same length as $[L_i,i]$, so $x'_{i+j} = x_{i+j} + \eps$. 
    
    If $x_j'\leq x_j$, then $x_i'+x_j'\leq x_{i+j}'$ as required. Otherwise, $x_j'=x_j+\eps$, so $j$ is the right endpoint of a segment. Since every super segment has length at least 2, it follows that $x_{i+j}\geq x_{j+1}+x_{i-1}>(x_j+x_1)+(x_i-x_1)=x_i+x_j$, so $(i,j,i+j)$ is not tight, a contradiction. This finishes Case 1.   

    \textbf{Case 2.} $x_i'\leq x_i$ and $x_j'=x_j+\eps$.

    If $x_{i+j}'=x_{i+j}+\eps$, then $x_{i+j}'\geq x_i'+x_j'$ and we are done, so assume otherwise. From definition, there are two ways we can have $x_j'=x_j+\eps$. If $[L_j,j]$ is a super segment, we proceed as in Case 1 above to reach a contradiction. 
    
    The additional possibility is that $r-j$ is the left endpoint of a super segment $[r-j,R]$ that is left but not right-crossing. Note that this is not possible in Case 1 as $i\leq k$. Then, $j$ is the right endpoint of a segment by Lemma~\ref{segsym}, so $L_{i+j}\geq j+1$. Since $x_{i+j}'\leq x_{i+j}$, $i+j$ is not the right endpoint of a super segment. Let $d=i+j-L_{i+j}<R-r+j$, so $x_{r+i-L_{i+j}+1}=x_{r-j+(d+1)}=x_{r-j}+(d+1)x_1$ and $x_{L_{i+j}-1}<x_{i+j}-(d+1)x_1$. Note that $i+j>L_{i+j}-1\geq j\geq r-k$, so it follows that $x_{r-L_{i+j}+1}=1-x_{L_{i+j}-1}>1-x_{i+j}+(d+1)x_1=x_{r-i-j}+(d+1)x_1$. Therefore,
    $x_i+x_{r-i-j}<x_i+x_{r-L_{i+j}+1}-(d+1)x_1=x_i+x_{r-L_{i+j}+1}+x_{r-j}-x_{r+i-L_{i+j}+1}\leq x_{r-j}$.
    It follows that $x_i+x_j<x_{i+j}$, so $(i,j,i+j)$ is not tight, a contradiction. This finishes Case 2.

    \textbf{Case 3.} $x'_{i+j} = x_{i+j} - \eps$, $x_i'\leq x_i$ and $x_j'\leq x_j$.
    
    If either $x_i'=x_i-\eps$ or $x_j'=x_j-\eps$, then $x_i'+x_j'\leq x_{i+j}'$, as required, so we assume both $x_i'=x_i$ and $x_j'=x_j$. Recall that $x_{i+j}'=x_{i+j}-\eps$ implies that $i+j$ is the left endpoint of a segment, so $x_{i+j-1}<x_{i+j}-x_1$. If $L_i<i$, then $x_{i-1}=x_i-x_1$, so $x_i+x_j=x_{i-1}+x_1+x_j\leq x_{i+j-1}+x_1<x_{i+j}$, and thus $(i,j,i+j)$ is not tight, a contradiction. Similarly, $L_j<j$ implies $(i,j,i+j)$ is not tight. Therefore, we may assume $L_i=i$ and $L_j=j$ from now on. From definition, we can have $x'_{i+j} = x_{i+j} - \eps$ in the following four ways.
    
    If $i+j$ is the left endpoint of a super segment and $i+j\leq k$, then we have $R_{i+j}-i-j\geq R_i-i$. If $R_{i+j}-i-j>R_i-i$, then $R_{i+j}-R_i>j$, so by Lemma \ref{moveright} and using $R_i<i+j\leq k$, $(i,j,i+j)$ is not tight, a contradiction. If $R_{i+j}-i-j=R_i-i$, then $[i,R_i]$ is also a super segment. Since $i\in[k]$, $x_i'=x_i-\eps$, a contradiction.

   If $i+j$ is the left endpoint of a super segment with $i+j\geq r-k$, then by Lemma \ref{segsym}, $[r-R_{i+j},r-i-j]$ is a super segment, and $r-i$ is the right endpoint of the segment $[L_{r-i},r-i]$. Since $[i+j,R_{i+j}]$ is a super segment, $R_{i+j}-i-j\geq r-i-L_{r-i}$. If $R_{i+j}-i-j=r-i-L_{r-i}$, then $[L_{r-i},r-i]$ is a super segment with $r-i\geq r-k$, so $x_i'=1-x_{r-i}'=1-(x_{r-i}+\eps)=x_i-\eps$ from definition, a contradiction. If $R_{i+j}-i-j>r-i-L_{r-i}$, then $L_{r-i}-(r-R_{i+j})>j$, so we can apply Lemma \ref{moveleft} to conclude that $(r-i-j, j, r-i)$ is not tight. Hence, $(i,j,i+j)$ is not tight, a contradiction.

   If $i+j$ is the left endpoint of a super segment that is right but not left-crossing, then we have $R_{i+j}-i-j\geq R_i-i$ and $R_{i+j}\geq r-k$. If $R_i\geq k$, then $i=r-R_{i+j}$ by Lemma \ref{segsym}, so $x_i'=x_i-\eps$ from definition, a contradiction. As such, we can assume $R_i\leq k-1$. If $R_{i+j}-i-j>R_i-i$, then we can apply Lemma \ref{moveright} to conclude that $(i,j,i+j)$ is not tight, a contradiction. If $R_{i+j}-i-j=R_i-i$, then $[i,R_i]$ is also a super segment, so $x_i'=x_i-\eps$ as $i\in[k]$, a contradiction. 

    The last possibility is $i+j=r-R$, where $[L,R]$ is a super segment that is right but not left-crossing. Then $R_i-i\leq R-L$, and by Lemma \ref{segsym}, $i+j$ is the left endpoint of a segment, so $R_i<i+j\leq k$. If $R_i-i=R-L$, then $[i,R_i]$ is a super segment, so $x_i'=x_i-\eps$, a contradiction. Otherwise, let $d=R_i-i<R-L$, then $x_{R-d-1}=x_R-(d+1)x_1$ and $x_{R_i+1}=x_{i+d+1}>x_i+(d+1)x_1$. It follows that $x_{r-R_i-1}<x_{r-i}-(d+1)x_1$ as $R_i+1\leq k$, and $x_j+x_{R-d-1}\leq x_{R-d+j-1}=x_{r-R_i-1}$ as $j\leq k$. Thus, $x_{r-i-j}=x_R=x_{R-d-1}+(d+1)x_1\leq x_{r-R_i-1}-x_j+(d+1)x_1<x_{r-i}-x_j$. Therefore, $x_i+x_j<x_{i+j}$ and $(i,j,i+j)$ is not tight, a contradiction. This finishes Case 3.

    If none of Case 1, Case 2 and Case 3 holds, then $x_i'\leq x_i$, $x_j'\leq x_j$ and $x_{i+j}'\geq x_{i+j}$, so $x_i'+x_j'\leq x_{i+j}'$. This proves $(x_1',\ldots,x_r')\in\mathcal{X}_{r,k}$.

    Finally, we show that $\prod_{i=1}^rx_i'>\prod_{i=1}^rx_i$. Note that $x_i=1-x_{r-i}$ and $x_i'=1-x_{r-i}'$ for all $i\in[k]$, and $x_i'\geq x_i$ for all $k+1\leq i\leq r-k-1$. Let $J=\{i\in[k]\mid x_i'\not=x_i\}$. Note that if we arrange indices $i\in J$ in increasing order, they alternate between $x_i'>x_i$ and $x_i'<x_i$, starting with $x_I'>x_I$. Pair up the consecutive indices in $J$ starting with $I$, with possibly $I'=\max J$ being left over if $|J|$ is odd, and let $P$ be the set of pairs obtained this way. By Lemma \ref{inequality} below, if $P\not=\emptyset$, then \[\prod_{(i,j)\in P}(x_i+\eps)(x_{r-i}-\eps)(x_j-\eps)(x_{r-j}+\eps)>\prod_{(i,j)\in P}x_ix_{r-i}x_jx_{r-j}\] for sufficiently small $\eps>0$. Similarly, if $|J|$ is odd, then $(x_{I'}+\eps)(1-x_{I'}-\eps)>x_{I'}(1-x_{I'})$ for sufficiently small $\eps>0$. Since $P=\emptyset$ implies that $|J|=1$ is odd, it follows that we always have $\prod_{i=1}^rx_i'>\prod_{i=1}^rx_i$, finishing the proof. 
\end{proof}

\begin{lemma}\label{inequality}
For every $0<a<b<1/2$ and $\eps>0$ sufficiently small, 
\[(a+\eps)(b-\eps)(1-a-\eps)(1-b+\eps)>ab(1-a)(1-b).\]
\end{lemma}
\begin{proof}
Let $f(t)$ be the polynomial $(a+t)(b-t)(1-a-t)(1-b+t)-ab(1-a)(1-b)$. Then $f(0)=0$, and the linear term of $f(t)$, which is equal to $f'(0)$, is \[b(1-a)(1-b)-a(1-a)(1-b)-ab(1-b)+ab(1-a)=(b-a)((1-a)(1-b)+ab)>0.\]
Therefore, for every sufficiently small $\eps>0$, $f(\eps)>0$, as required.
\end{proof}

We can now prove Theorem \ref{thm:main}.
\begin{proof}[Proof of Theorem \ref{thm:main}]
Recall that we may assume $k=\ceil{r/e}\leq\floor{r/2}-1$, as the case when $k=\floor{r/2}$ is proved by Theorem 7.1 in \cite{CY}. It is easy to check that if $x_i=i/r$ for every $i\in[r]$, then $(x_1,\ldots,x_r)\in\mathcal{X}_{r,k}$ and $\prod_{i=1}^r=r!/r^r$, so the maximum is at least $r!/r^r$. 

Now suppose that $(x_1,\ldots,x_r)\in\mathcal{X}_{r,k}$ maximises $\prod_{i=1}^rx_i$, by Lemma \ref{I<=k}, we may assume that $I\geq k$. Since $x_1+x_i\leq x_{i+1}$ for all $i\in[r-1]$, we have $rx_1\leq x_r=1$ and thus $x_1\leq1/r$. Let $x_1=(1-\eps)/r$ for some $\eps\geq 0$. It follows from $I\geq k$ that $[1,k]$ is uniform, so $x_i=ix_1=i(1-\eps)/r$ for all $i\in[k]$. Moreover, for every $k+1\leq i\leq r$, we have $(r-i)x_1+x_i\leq x_r=1$, so $x_i\leq 1-(r-i)x_1 =(i+(r-i)\eps)/r$.

To prove $\prod_{i=1}^rx_i\leq r!/r^r=\prod_{i=1}^r\frac ir$, it suffices to show
\[\log\left(\frac{\prod_{i=1}^rx_i}{\prod_{i=1}^r\frac ir}\right) \leq \log\left(\prod_{i=1}^k(1-\eps)\prod_{i=k+1}^r\left(1+\frac{(r-i)\eps}{i}\right)\right)\leq 0.\]
Indeed, using $\log(1+t)\leq t$ for all $t\in\mathbb{R}$, we have
\begin{align*}
\log\left(\prod_{i=1}^k(1-\eps)\prod_{i=k+1}^r\left(1+\frac{(r-i)\eps}{i}\right)\right)&=\sum_{i=1}^k\log(1-\eps)+\sum_{i=k+1}^r\log\left(1+\frac{(r-i)\eps}{i}\right)\\
&\leq-k\eps+\sum_{i=k+1}^r\frac{(r-i)\eps}{i}\leq\eps\left(r\int_{\frac kr}^1\frac{1-t}t\text{d}t-k\right)\\
&=\eps\left(r\left(-1-\log\left(\frac kr\right)+\frac kr\right)-k\right)\\
&=\eps r(\log r-\log k-1)\leq 0,
\end{align*}
where the last inequality follows from $r\leq ke$. This proves the upper bound and shows the maximum is exactly $r!/r^r$. In fact, we have $r<ke$ since $e$ is irrational, so the last inequality is above is strict unless $\eps=0$. In that case, from the proof above we have $x_i\leq i/r$ for all $i\in[r]$, so the maximum value $r!/r^r$ is achieved if and only if $x_i=i/r$ for all $i\in[r]$.
\end{proof}

Finally, we show that the choice of $\ceil{r/e}$ in Theorem~\ref{thm:main} is essentially best possible, so the choice of $\ceil{r/e}$ in Theorem~\ref{thm:maindensity} cannot be improved using this entropy method.
\begin{lemma}\label{counterexample}
For every $r\geq2$, if $1\leq k<\floor{r/e}$, then $\prod_{i=1}^rx_i>r!/r^r$ for some $(x_1,\ldots,x_r)\in\mathcal{X}_{r,k}$.
\end{lemma}
\begin{proof}
Let $\eps>0$ be sufficiently small. Define $(x_1,\ldots,x_r)$ by
\[x_i=\begin{cases}
\frac ir-\frac{i\eps}r, &\text{if }i\in[k],\\
\frac ir+\frac{(r-i)\eps}r, &\text{if }k+1\leq i\leq r.
\end{cases}\]
It is clear that for all sufficiently small $\eps$, we have $0<x_1\leq\cdots\leq x_r=1$. For all $1\leq i<j\leq k$, $x_i+x_j=(1-\eps)(i+j)/r\leq x_{i+j}$. For all $1\leq i\leq k<j\leq r$ that satisfy $i+j\leq r$, we have $x_i+x_j=(i+j+(r-i-j)\eps)/r=x_{i+j}$. Hence, $(x_1,\ldots,x_r)\in\mathcal{X}_{r,k}$.

To prove $\prod_{i=1}^rx_i>r!/r^r=\prod_{i=1}^r\frac ir$, it suffices to show
\[\frac{\prod_{i=1}^rx_i}{\prod_{i=1}^r\frac ir}=\prod_{i=1}^k(1-\eps)\prod_{i=k+1}^r\left(1+\frac{(r-i)\eps}{i}\right)>1.\]
Define \[f(t)=\prod_{i=1}^k(1-t)\prod_{i=k+1}^r\left(1+\frac{(r-i)t}{i}\right).\]
Then $f(t)$ is a polynomial in $t$ with $f(0)=1$. The linear term of $f(t)$, which is equal to $f'(0)$, is
\begin{align*}
-k+\sum_{i=k+1}^r\frac{(r-i)}{i}&=-k+r\sum_{i=k+1}^{r-1}\frac{1-\frac ir}{\frac ir}\cdot\frac1r\geq-k+r\int_{\frac{k+1}r}^1\frac{1-t}t\text{d}t\\
&=-k+r\left(-1-\log\left(\frac{k+1}r\right)+\frac{k+1}r\right)\\
&=1+r(\log r-\log(k+1)-1)>0,
\end{align*}
as $r\geq(k+1)e$. Hence, for sufficiently small $\eps>0$, $f(\eps)>f(0)=1$, which finishes the proof.
\end{proof}

\section{Proof of Theorem \ref{thm:mainturan}}\label{sec:turan}
In this section, we prove Theorem \ref{thm:mainturan}. This is done by combining Theorem \ref{thm:maindensity} and the machinery developed in \cite{HLZ24}, \cite{L} and \cite{LMR23unif}. We proceed similarly as in \cite{L} where Liu obtained Theorem \ref{thm:Lmain} from Theorem \ref{thm:CYmain}. We begin by introducing the necessary concepts.

\begin{definition}
Let $H$ be an $r$-uniform hypergraph. Let $\mathcal{F}$ be a family of $r$-uniform hypergraphs. Let $\mathfrak{H}$ be a family of $\mathcal{F}$-free $r$-uniform hypergraphs.  
\begin{itemize}
    \item $H$ is \textit{2-covered} if every pair of vertices is contained in at least one edge of $H$.
    \item $H$ is \textit{symmetrised} if $H$ is the blowup of a 2-covered $r$-uniform hypergraph.
    \item $\mathcal{F}$ is \textit{blowup-invariant} if every blowup of an $\mathcal{F}$-free $r$-uniform hypergraph $F$ is also $\mathcal{F}$-free.
    \item $\mathcal{F}$ is \textit{hereditary} if for every $F\in\cF$, every subgraph $F'$ of $F$ is in $\cF$ as well.
    \item $\mathcal{F}$ is \textit{symmetrised-stable} with respect to $\mathfrak{H}$ if every symmetrised $\mathcal{F}$-free $r$-uniform hypergraph is contained in $\mathfrak{H}$.
\end{itemize}
Now assume additionally that $\pi(\mathcal{F})>0$. 
\begin{itemize}
    \item $\mathcal{F}$ is \textit{edge-stable} with respect to $\mathfrak{H}$ if for every $\delta>0$ there exist $\varepsilon>0$ and $N$ such that every $\mathcal{F}$-free $r$-uniform hypergraph $H$ on $n \ge N$ vertices with $|E(H)| \ge \left(\pi(\cF)/r! - \varepsilon\right)n^r$ can be turned into a member in $\mathfrak{H}$ by removing at most $\delta n^r$ edges. 
    \item $\mathcal{F}$ is \textit{degree-stable} with respect to $\mathfrak{H}$ if there exist $\varepsilon>0$ and $N$ such that every $\mathcal{F}$-free $r$-uniform hypergraph $H$ on $n \ge N$ vertices with $\delta(H) \ge \left(\pi(\cF)/(r-1)! - \varepsilon\right)n^{r-1}$ is in $\mathfrak{H}$. 
    \item $\mathcal{F}$ is \textit{vertex-extendable} with respect to $\mathfrak{H}$ if there exist $\varepsilon>0$ and $N$ such that for every $\mathcal{F}$-free $r$-uniform hypergraph $H$ on $n \ge N$ vertices with $\delta(H) \ge \left(\pi(\cF)/(r-1)! - \varepsilon\right)n^{r-1}$, if $H-v$ is in $\mathfrak{H}$, then $H$ is in $\mathfrak{H}$ as well. 
\end{itemize}
\end{definition}

These notions were introduced in~\cite{LMR23unif} to provide a unified framework for proving the degree-stability of certain classes of hypergraph families. This method was later refined in~\cite{HLZ24}. Their main stability results are the following. 
\begin{theorem}[{\cite[Theorem~1.7]{LMR23unif}} and {\cite[Theorem~1.1]{HLZ24}}]\label{THM:LMR-vtx-extend}
    Suppose that $\mathcal{F}$ is a family of $r$-uniform hypergraphs with $\pi(\mathcal{F})>0$, and $\mathfrak{H}$ is a hereditary family of $\mathcal{F}$-free $r$-uniform hypergraphs.
    Then the following statements hold.
    \begin{enumerate}[label=(\roman*)]
        \item\label{THM:LMR-vtx-extend-1} If $\mathcal{F}$ is blowup-invariant, and both symmetrised-stable and vertex-extendable with respect to $\mathfrak{H}$, then $\mathcal{F}$ is degree-stable with respect to $\mathfrak{H}$.
        \item\label{THM:LMR-vtx-extend-2} If $\mathcal{F}$ is both edge-stable and vertex-extendable with respect to $\mathfrak{H}$, then $\mathcal{F}$ is degree-stable with respect to $\mathfrak{H}$.
    \end{enumerate}
\end{theorem}

Since $\cF_{r,k}$ is not blowup-invariant, for part of the proof we need to consider the following broader family of triangle-like $r$-uniform hypergraphs. Denote by $\mathcal{T}_{r, k}$ the collection of $r$-uniform hypergraphs consisting of three edges $A, B, C$ such that 
\begin{align*}
    A \subset B \cup C,
    \quad\text{and}\quad 
    (B\cap C) \setminus A \neq \emptyset
    \quad\text{and}\quad 
    |A\cap B| \geq r - k.
\end{align*}
Note that indeed $\mathcal{F}_{r, k}\subseteq \mathcal{T}_{r, k}$, so $\pi(\mathcal{T}_{r, k})\leq\pi(\mathcal{F}_{r, k})=r!/r^r$ by Theorem \ref{thm:maindensity}. But $\pi(\mathcal{T}_{r,k})\geq r!/r^r$ because it is at least the blowup density of a single edge, so $\pi(\mathcal{T}_{r,k})=r!/r^r$. Moreover, one can verify that $\mathcal{T}_{r, k}$ is blowup-invariant. To prove Theorem~\ref{thm:mainturan} we need following three propositions.

\begin{proposition}\label{PROP:min-deg-r-partite}
    For every $r\geq 4$, let $k=\ceil{r/e}$, then there exist $\varepsilon>0$ and $N$ such that the following holds for all $n \ge N$. Suppose that $H$ is the subgraph of a symmetrised $\mathcal{T}_{r,k}$-free $r$-uniform hypergraph on $n$ vertices with $\delta(H) \ge n^{r-1}/r^{r-1} - \varepsilon n^{r-1}$. Then $H$ is $r$-partite. 
\end{proposition}
\begin{proof}
    By Theorem~\ref{thm:main}, if $(x_1,\ldots,x_r)\in\mathcal{X}_{r,k}$ maximizes $\prod_{i=1}^rx_i$, then $x_i = i/r$ for $1\leq i \leq r$. We can then prove an analogous version of Proposition 3.1 in~\cite{L}, with $\mathcal{T}_{r,k}$ replacing the $\mathcal{T}_r$ there and using the same proof. Then we proceed as in Proposition 5.1 in~\cite{L} to prove an analogous version again with $\mathcal{T}_{r,k}$ replacing $\mathcal{T}_r$, which gives exactly Proposition~\ref{PROP:min-deg-r-partite}.
\end{proof}

\begin{proposition}[{\cite[Proposition~4.1]{L}}]\label{PROP:Tr-vtx-extend}
For every $r\geq 2$, $\Delta_{(r-1,1)}$ is vertex-extendable with respect to $\mathfrak{K}_{r}^{r}$, the family of all $r$-partite $r$-uniform hypergraphs.  
\end{proposition}

\begin{proposition}\label{prop:deg-stable}
For every $r\geq 4$, let $k=\ceil{r/e}$, then $\mathcal{F}_{r,k}$ is degree-stable with respect to $\mathfrak{K}_{r}^{r}$.
\end{proposition}
\begin{proof}
Let $\mathfrak{S}$ denote the collection of all $2$-covered $\mathcal{T}_{r, k}$-free $r$-uniform hypergraphs. Let $\mathfrak{H}$ be the collection of every $r$-uniform hypergraph $H$ that is the subgraph of a blowup of some $G\in\mathfrak S$. In particular, $\mathfrak H$ contains every blowup of every $G\in\mathfrak S$, as well as the collection $\mathfrak{K}_{r}^{r}$ of all $r$-partite $r$-uniform hypergraphs.

We first prove that $\mathcal{T}_{r, k}$ is degree-stable with respect to $\mathfrak{H}$. It follows from the definition that $\mathfrak{H}$ is hereditary and $\mathcal{T}_{r, k}$ is symmetrised-stable with respect to $\mathfrak H$. Also, $\pi(\mathcal{T}_{r, k})=r!/r^r>0$ and $\mathcal{T}_{r, k}$ is blowup-invariant, so by Theorem~\ref{THM:LMR-vtx-extend}~\ref{THM:LMR-vtx-extend-1}, it suffices to show that $\mathcal{T}_{r, k}$ is vertex-extendable with respect to $\mathfrak{H}$. 

Let $\varepsilon > 0$ be a sufficiently small constant and let $n$ be a sufficiently large integer. Let $H$ be a $\mathcal{T}_{r, k}$-free $r$-uniform hypergraph on $n$ vertices with $\delta(H) \ge n^{r-1}/r^{r-1} - \varepsilon n^{r-1}$. Suppose that $v^\ast \in V(H)$ is a vertex such that $H - v^\ast$ is in $\mathfrak{H}$. Then, as
\[\delta(H-v^\ast)\ge \delta(H) - \binom{n-2}{r-2}\ge n^{r-1}/r^{r-1} - 2 \varepsilon n^{r-1},\]
we have that $H-v^\ast\in \mathfrak{K}_{r}^{r}$ by Proposition~\ref{PROP:min-deg-r-partite}. By Proposition~\ref{PROP:Tr-vtx-extend}, $\Delta_{(r-1,1)} \in \mathcal{T}_{r, k}$ is vertex-extendable with respect to $\mathfrak{K}_{r}^{r}$, so $\mathcal{T}_{r, k}$ is as well. Therefore, by definition, $H$ is in $\mathfrak{K}_{r}^{r} \subseteq \mathfrak{H}$, which proves that $\mathcal{T}_{r, k}$ is vertex-extendable with respect to $\mathfrak{H}$, and thus degree-stable with respect to $\mathfrak{H}$ by Theorem~\ref{THM:LMR-vtx-extend}~\ref{THM:LMR-vtx-extend-1}.

By Proposition~\ref{PROP:min-deg-r-partite} again, $\mathcal{T}_{r, k}$ being degree-stable with respect to $\mathfrak{H}$ implies that it is also degree-stable with respect to $\mathfrak{K}_{r}^{r}$. This then implies that $\mathcal{T}_{r,k}$ is edge-stable with respect to $\mathfrak{K}_{r}^{r}$ by removing vertices of low degrees (see e.g. \cite[Fact 2.5]{LMR23unif}). Since for every $F \in \mathcal{T}_{r, k}$, there exists $F' \in \mathcal{F}_{r,k}$ such that there is a homomorphism from $F'$ to $F$, a standard application of the Hypergraph Removal Lemma (see e.g.~\cite{Gow07,NRS06,RS04}) shows that every $\mathcal{F}_{r,k}$-free $r$-uniform hypergraph on $n$ vertices can be made $\mathcal{T}_{r,k}$-free by removing $o(n^r)$ edges. This implies that $\mathcal{F}_{r,k}$ is edge-stable with respect to $\mathfrak{K}_{r}^{r}$ as well. By Proposition~\ref{PROP:Tr-vtx-extend} again, $\Delta_{(r-1,1)}$ and therefore $\mathcal{F}_{r,k}$ is vertex-extendable with respect to $\mathfrak{K}_{r}^{r}$, so it follows from Theorem~\ref{THM:LMR-vtx-extend}~\ref{THM:LMR-vtx-extend-2} that $\mathcal{F}_{r,k}$ is degree-stable with respect to $\mathfrak{K}_{r}^{r}$.
\end{proof}

Now we can bring everything together and prove Theorem \ref{thm:mainturan}.
\begin{proof}[Proof of Theorem \ref{thm:mainturan}]
For simplicity, let $|E(T^r(n))|=e_n$ for every $n\in\mathbb{Z}^+$. By Proposition \ref{prop:deg-stable}, there exists $\eps, N$ such that every $\mathcal{F}_{r,k}$-free $r$-uniform hypergraph $H$ on $n\geq N$ vertices with $\delta(H)\geq n^{r-1}/r^{r-1}-\eps n^{r-1}$ is $r$-partite. 

Now let $N_0\gg N$ satisfy $\ex(n,\mathcal{F}_{r,k})\leq(1+\eps)n^r/r^r$ for all $n\geq N_0$, and let $H$ be an $\mathcal{F}_{r,k}$-free $r$-uniform hypergraph on $n\geq 2N_0$ vertices with $|E(H)|\geq|E(T^r(n))|$. Define a sequence of subgraphs $H=H_0\supset H_1\supset\cdots$ with $|V(H_i)|=n_i=n-i$ for all $i\geq 0$ by repeatedly deleting a single vertex (as long as it exists) from $H_i$ with degree at most $e_{n_i}-e_{n_{i+1}}-\eps n_i^{r-1}/2$ to obtain $H_{i+1}$. This process must terminate after at most $n/10$ steps as otherwise for $i=\ceil{n/10}$, $n_i\geq N_0$ and \[|E(H_i)|\geq e_{n_i}+i\eps n_i^{r-1}/2\geq n_i^r/r^r+\eps n_i^r/20>(1+\eps)n_i^r/r^r,\]
contradicting $\ex(n_i,\mathcal{F}_{r,k})\leq(1+\eps)n_i^r/r^r$. Suppose now that this process terminated after $t$ steps for some $t\leq n/10$, then $H_t$ satisfies $|E(H_t)|\geq e_{n_t}+t\eps n_t^{r-1}/2$, $n_t\geq N$ and \[\delta(H_t)\geq e_{n_t}-e_{n_t-1}-\eps n_t^{r-1}/2\geq\floor{n_t/r}^{r-1}-\eps n_t^{r-1}/2\geq n_t^{r-1}/r^{r-1}-\eps n_t^{r-1}.\] Thus, $H_t$ is $r$-partite from degree-stability, so $|E(H_t)|\leq|E(T^r(n_t))|=e_{n_t}$, which gives a contradiction unless $t=0$. In that case $H_0=H$ is $r$-partite and so $|E(H)|\leq|E(T^r(n))|$, with equality if and only if $H=T^r(n)$.
\end{proof}

\section{Concluding Remarks}\label{sec:conclusion}
We begin by posing two follow up questions. By Lemma~\ref{counterexample}, it is not possible to use this entropy method to show $\pi(\mathcal F_{r,k})=r!/r^r$ for $k<\floor{r/e}$. For $k=\floor{r/e}$, careful examination of our proof shows that whether this method works depends on how close $r$ is to an integer multiple of $e$. The natural question then is whether $\pi(\mathcal F_{r,k})>r!/r^r$ if $k<\floor{r/e}$, so that our result obtained from the entropy method is essentially tight.

\begin{question}\label{ques}
    Is $\pi(\mathcal F_{r,k})>r!/r^r$ if $k<\floor{r/e}$? If not, what is the largest $k$ such that this holds?
\end{question}

Also, in this paper we considered only subfamilies of $\mathcal{F}_{r,\floor{r/2}}$ of the form $\mathcal F_{r,k}$, but one could consider more general subfamilies of $\mathcal{F}_{r,\floor{r/2}}$. Therefore, it would interesting to know the following.

\begin{question}
    What are the subfamilies $\mathcal F \subset \mathcal{F}_{r,\floor{r/2}}$ that satisfy $\pi(\cF) = r!/r^r$?
\end{question}

Finally, we mention two related problems, for which our results have some implications. 
\subsection{Generalised tents}
More generally, for any partition $\lambda=(\lambda_1,\ldots,\lambda_m)\vdash r$ with $r>\lambda_1\geq\cdots\geq\lambda_m$, the \textit{$\lambda$-tent} $\Delta_\lambda$ is defined to be the unique $r$-uniform hypergraph up to isomorphism with $m+1$ edges $e_0,e_1,\ldots,e_m$, such that
\begin{itemize}
    \item $|e_0\cap e_i|=\lambda_i$ for every $i\in[m]$,
    \item $e_0\cap e_1,\ldots,e_0\cap e_m$ partition $e_0$,
    \item there is a vertex $v$ such that $e_i\cap e_j=\{v\}$ for every $1\leq i<j\leq m$.
\end{itemize}
Note that when $m=2$, this coincides with our earlier definition of the $(r-i,i)$-tent $\Delta_{(r-i,i)}$. 

In \cite{MP07}, Mubayi and Pikhurko showed that $\pi(\Delta_{(1,1,\ldots,1)})=r!/r^r$, and for large $n$, $\ex(n,\Delta_{(1,1,\ldots,1)})=|E(T^r(n))|$, with the $T^r(n)$ being the unique extremal graph. Since there is a homomorphism from $\Delta_{(1,1,\ldots,1)}$ to $F$ for every $F\in\mathcal{F}_{r,\floor{r/2}}$, we have $\pi(\Delta_{(1,1,\ldots,1)})\leq\pi(\mathcal{F}_{r,\floor{r/2}})$ by Corollary~\ref{lemma:densityineq}, so Theorem \ref{thm:CYmain} is stronger than Mubayi and Pikhurko's result. Moreover, as there is a homomorphism from $\Delta_{(\ceil{r/2},1,\ldots,1)}$ to $F$ for every $F\in\mathcal{F}_{r,\floor{r/2}}$, Theorem \ref{thm:CYmain} also implies that $\pi(\Delta_{(\ceil{r/2},1,\ldots,1)})=r!/r^r$. Using the same argument, we obtain the following improvement as a corollary to Theorem \ref{thm:maindensity}. 
\begin{corollary}\label{cor:maindensity}
For every $r\geq4$, let $k=\ceil{r/e}$, then $\pi(\Delta_{(r-k,1,\ldots,1)})=r!/r^r$.
\end{corollary}
\begin{proof}
This follows immediately from Theorem \ref{thm:maindensity} and Corollary~\ref{lemma:densityineq}, using the observation that there is a homomorphism from $\Delta_{(r-k,1,\ldots,1)}$ to $F$ for every $F\in\mathcal{F}_{r,k}$.
\end{proof}

\subsection{$L$-intersecting hypergraphs}
Let $L\subset\{0,1,\ldots,r-1\}$. An $r$-uniform hypergraph $H$ is \textit{$L$-intersecting} if for any distinct $e,e'\in E(H)$, $|e\cap e'|\in L$. Determining the maximum blowup density of an $L$-intersecting $r$-uniform hypergraph is an interesting open problem, and is related to Tur\'an number of family of tents because of the following simple lemma.
\begin{lemma}\label{lemma:Lintertent}
Let $\mathcal{F}$ be a family of tents satisfying that $\lambda\vdash r$ and $r>\lambda_1\geq k$ for all $\Delta_\lambda\in\mathcal{F}$. Let $L\subset\{0,1,\ldots,k-1\}$. Then for every $L$-intersecting $r$-uniform hypergraph $H$, $|E(H)|\leq\ex(|H|,\mathcal{F})$ and $b(H)\leq\pi(\mathcal{F})$.
\end{lemma}
\begin{proof}
Since $H$ is $L$-intersecting and every $\Delta_\lambda\in\mathcal{F}$ contains two edges whose intersection has size at least $k$, $H$ is $\mathcal{F}$-free, so $|E(H)|\leq\ex(|H|,\mathcal{F})$. 

Now let $\Delta_\lambda\in\mathcal{F}$ be a tent with edges $e_0,\ldots,e_m$ and $\cap_{i=1}^me_i=\{v\}$ as in the definition, and suppose there is a homomorphism $f:\Delta_\lambda\to H$. Then, $|e_0\cap e_1|=\lambda_1\geq k$ and $H$ is $L$-intersecting implies that $f(e_0)=f(e_1)$, so in particular, $f(v)=f(u)$ for some $u\in e_0$. But both $u$ and $v$ are in some edge $e_i$ for some $i\in[m]$, so $f$ is not injective on $e_i$, a contradiction. Hence, $H$ is $\mathcal{F}$-hom-free, and the result follows from Lemma \ref{lemma:homdensity}.
\end{proof}
Using this argument, Liu \cite{L} observed that Theorem \ref{thm:CYmain} implies that if $L\subset\{0,1,\ldots,\ceil{r/2}-1\}$ and $H$ is $r$-uniform and $L$-intersecting, then $b(H)\leq r!/r^r$. We can improve this as follows.
\begin{corollary}\label{cor:Lintersect}
For every $r\geq4$, let $k=\ceil{r/e}$, let $L\subset\{0,1,\ldots,r-k-1\}$. If $H$ is an $L$-intersecting $r$-uniform hypergraph, then $b(H)\leq r!/r^r$.
\end{corollary}
\begin{proof}
This is the immediate consequence of Theorem \ref{thm:maindensity} and Lemma \ref{lemma:Lintertent}.
\end{proof}

In the other direction, Lemma \ref{lemma:Lintertent} also provides a possible approach to Question \ref{ques}. For a certain $\alpha\approx0.89$, $L=\{0,1,\ldots,\floor{\alpha r}\}$ and $r$ large, Chao and Yu constructed in \cite{CY} an $L$-intersecting $r$-uniform hypergraph $H$ on $2r$ vertices with $b(H)>r!/r^r$. By Lemma \ref{lemma:Lintertent}, this implies that for all sufficiently large $r$ and any family $\mathcal{F}$ of tents satisfying $\lambda\vdash r$ and $r>\lambda_1>\alpha r$ for all $\Delta_\lambda\in\mathcal F$, we have $\pi(\mathcal{F})\geq b(H)>r!/r^r$. In particular, the answer to the second part of Question \ref{ques} is asymptotically at least $(1-\alpha)r$. It is an open question to determine the smallest $\alpha$ for which a construction like Chao and Yu's exists. Note that Liu's observation above implies that this is at least 1/2, and Corollary \ref{cor:Lintersect} improves this lower bound to $1-1/e$.

\section*{Acknowledgement}
We would like to thank Xizhi Liu for telling us about \cite{CY} and \cite{L}, and for related discussions.

\bibliographystyle{plain}
\bibliography{bibliography}

\end{document}